\documentclass[12pt]{article}
\usepackage{amsmath, amsthm, amscd, amsfonts, amssymb, graphicx}

\usepackage{upgreek}
\usepackage{color}
\usepackage{tikz-cd}

\usepackage[margin=0.9in]{geometry}

\usepackage{microtype}
\usepackage{hyperref}
\usepackage[capitalise]{cleveref}
\hypersetup{
	colorlinks,
	linkcolor={red!70!black},
	citecolor={blue!80!black},
	urlcolor={blue!80!black}
}
\usepackage{enumerate}
\usepackage[breakable,skins]{tcolorbox}
\usepackage[normalem]{ulem}
\usepackage[
backend=biber,
style=numeric,
maxnames=99,
maxbibnames=99,
giveninits=true,
doi=false,
isbn=false,
url=false,
eprint=false
]{biblatex}
\DeclareFieldFormat[article,inbook,incollection,inproceedings,thesis,unpublished,misc]{title}{#1\isdot}

\renewbibmacro{in:}{}
\definecolor{humancolor}{RGB}{0, 102, 204}    
\definecolor{aicolor}{RGB}{180, 90, 50}       

\newlength{\marginbarwidth}
\newlength{\marginbarsep}
\makeatletter
\newtoks\HAI@footnotes

\newcommand{\haifootnote}[1]{%
  \footnotemark
  \global\HAI@footnotes=\expandafter{\the\HAI@footnotes\footnotetext{#1}}%
}

\newcommand{\HAI@emitfootnotes}{%
  \the\HAI@footnotes
  \global\HAI@footnotes={}%
}
\makeatother

\makeatletter
\newtcolorbox{humanparagraph}{%
  enhanced,
  breakable,
  blanker,
  borderline west={\marginbarwidth}{-\marginbarsep}{humancolor},
  left=0pt,
  right=0pt,
  top=2pt,
  bottom=2pt,
  before skip=0.5\baselineskip,
  after skip=0.5\baselineskip,
  after={\HAI@emitfootnotes},
  overlay unbroken and first={
    \node[humancolor, font=\tiny\scshape, rotate=90, anchor=north]
      at ([xshift=-\marginbarsep-\marginbarwidth/2-12pt, yshift=-13pt]frame.north west) {Human};
  }
}

\newtcolorbox{aiparagraph}{%
  enhanced,
  breakable,
  blanker,
  borderline west={\marginbarwidth}{-\marginbarsep}{aicolor},
  left=0pt,
  right=0pt,
  top=2pt,
  bottom=2pt,
  before skip=0.5\baselineskip,
  after skip=0.5\baselineskip,
  after={\HAI@emitfootnotes},
  overlay unbroken and first={
    \node[aicolor, font=\tiny\scshape, rotate=90, anchor=north]
      at ([xshift=-\marginbarsep-\marginbarwidth/2-12pt, yshift=-5pt]frame.north west) {AI};
  }
}
\makeatother

\newcommand{\defend}{\hfill\mbox{$\lozenge$}}

\theoremstyle{plain} 

\newtheorem{question}{Question}

\newtheorem{theorem}{Theorem}[section] 
\newtheorem{corollary}[theorem]{Corollary} 
\newtheorem{lemma}[theorem]{Lemma}         
\newtheorem{proposition}[theorem]{Proposition} 

\theoremstyle{definition} 
\newtheorem{defin}[theorem]{Definition} 

\newtheorem{remark}[theorem]{Remark}       

\AddToHook{env/corollary/begin}{\crefalias{theorem}{corollary}}
\AddToHook{env/lemma/begin}{\crefalias{theorem}{lemma}}
\AddToHook{env/proposition/begin}{\crefalias{theorem}{proposition}}
\AddToHook{env/defin/begin}{\crefalias{theorem}{defin}}
\AddToHook{env/remark/begin}{\crefalias{theorem}{remark}}

 \AtEndEnvironment{defin}{\defend}
 \AtEndEnvironment{remark}{\defend}

\usepackage{hyperref}
\usepackage{xcolor}
\usepackage{color,soul}
\usepackage{mathtools}

\newcommand{\HDpoly}{\chi_c}
\newcommand{\Q}{\mathbb{Q}}
\newcommand{\Z}{\mathbb{Z}}
\newcommand{\C}{\mathbb{C}}

\newcommand{\Pic}{\mathrm{Pic}}

\newcommand{\uul}{\underline{C}}
\newcommand{\Jcal}{\mathcal{J}}
\newcommand{\Mcal}{\mathcal{M}}
\newcommand{\Mbar}{\overline{\mathcal{M}}}
\newcommand{\Jbar}{\overline{\mathcal{J}}}

\newcommand{\Aut}{\mathrm{Aut}}

\title{Universal compactified Jacobians: cohomological invariance and boundary combinatorics}
\author{Rahul Pandharipande, Dan Petersen, Johannes Schmitt, Sofia Wood\\with an Appendix by Jeremy Feusi and Qizheng Yin}
\date{July 2026}
\begin{document}

\maketitle

 \begin{abstract}
Pagani and Tommasi have introduced a class of smoothable fine compactified Jacobians
 $\overline{\mathcal{J}}_{g,n}^d(\sigma)\rightarrow \overline{\mathcal{M}}_{g,n}$ over the moduli space of stable curves, depending nontrivially on the degree $d$ and the choice of a stability condition $\sigma$. A theorem of Migliorini--Shende--Viviani implies that the cohomology of $\overline{\mathcal{J}}_{g,n}^d(\sigma)$ is independent of $d$ and $\sigma$, a statement which is quite unexpected from the point of view of the boundary geometry of these spaces. We reprove this  independence statement using
 a direct combinatorial argument, summing up contributions of individual strata. 

The Appendix includes a result by J.~Feusi characterizing when ${\mathcal{J}}_{g,n}^d$
and ${\mathcal{J}}_{g,n}^{d'}$ are $S_n$-equivariantly isomorphic over $\mathcal{M}_{g,n}$, and
a result by Q.~Yin showing that $[\mathcal{J}^d_g]$ and $[\mathcal{J}^{d'}_g]$ are not always equal
in $K_0(\text{Var}_{\mathbb{C}})$.
 

\end{abstract}

\tableofcontents
\section{Introduction}
\subsection{Universal Jacobians}\label{intro: uncompactified}
Let $g$ and $n$ be non-negative integers satisfying $2g-2+n>0$. For every $d \in \mathbb{Z}$, 
there is a universal Jacobian
\begin{equation}
\mathcal{J}_{g,n}^d = \left\{(C, p_1, \ldots, p_n, L) \ \Bigg|\ 
\begin{array}{c}
C \text{ is a nonsingular curve of genus }g, \\ p_1, \ldots, p_n \in C \text{ are pairwise distinct},\\
L/C \text{ is a line bundle of degree }d
\end{array}
\right\}{\Big/\sim }
\end{equation}
over the moduli space $\mathcal{M}_{g,n}$ of nonsingular $n$-pointed genus $g$ curves. 

When $n>0$, the universal Jacobians $\mathcal J_{g,n}^d$ are all isomorphic over $\mathcal M_{g,n}$ as $d$ varies. Indeed, such isomorphisms are for example given by
\begin{equation}
    \label{twist map} \mathcal J_{g,n}^{d} \ni (C,p_1,\dots,p_n,L) \longmapsto (C,p_1,\dots,p_n,L \otimes \mathcal O((d'-d)p_n)) \in \mathcal J_{g,n}^{d'}.
\end{equation} The situation when $n=0$ is in sharp contrast: by \cite[Lemma 8.1]{Cap94}, the
universal Jacobians 
$\mathcal{J}_{g}^d$ and
$\mathcal{J}_{g}^{d'}$ are isomorphic over $\mathcal{M}_g$
if and only if 
$$d\equiv d' \quad \text{or}\quad {-d}\equiv d'
\mod 2g-2\, $$ with isomorphisms given by $$L \mapsto L \otimes \omega_C^{\otimes k} \quad \text{or} \quad L \mapsto L^{-1} \otimes \omega_C^{\otimes k},$$ for some $k$, respectively. When $g \geq 1$ and $n>1$, the isomorphism \eqref{twist map} is $S_{n-1}$-equivariant, but not $S_n$-equivariant. Since only twisting by all the points simultaneously respects the 
$S_n$-symmetry,  the
parallel claim{\footnote{The proof, due to Jeremy Feusi, is presented in \cref{Sect:Misc}. The main step of the argument can already be found in \cite[Corollary 6.1]{KP19}.}} is that the universal Jacobians 
$\mathcal{J}_{g,n}^d$ and
$\mathcal{J}_{g,n}^{d'}$ are $S_n$-equivariantly isomorphic over $\mathcal{M}_{g,n}$
if and only if 
$$d\equiv d' \quad \text{or}\quad {-d}\equiv d'
\mod \gcd(2g-2,n)\, .$$
The above differences are completely invisible on the level of cohomology. Indeed, one has the following:
\begin{proposition}\label{uncompactified theorem}
    For all $2g-2+n>0$ and all $d,d' \in \Z$, there exist isomorphisms
    \begin{eqnarray*} H^*(\mathcal J_{g,n}^{d},\Q) & \cong & H^*(\mathcal J_{g,n}^{d'},\Q)\,, \\
    H^*_c(\mathcal J_{g,n}^{d},\Q) & \cong & H^*_c(\mathcal J_{g,n}^{d'},\Q) \end{eqnarray*}
    compatible with $S_n$-action, with mixed Hodge structure, and with cup-product. 
\end{proposition}
The above result is well-known to experts, 
see \cite{BaeMSY}. We give another proof in \cref{rigidity section}. \cref{uncompactified theorem} cannot be lifted to the Grothendieck group of varieties:
the classes of $\mathcal{J}_{g}^d$ are {\em dependent} upon $d$ in general.{\footnote{The proof, due to Qizheng Yin, is presented in \cref{Sect:Misc}.}}

\subsection{Compactified universal Jacobians}

A classical problem in algebraic geometry is to extend the family  
$$\pi: \mathcal{J}_{g,n}^d \to \mathcal{M}_{g,n}$$ over the moduli space of
stable curves $\overline{\mathcal{M}}_{g,n}$ 
while preserving as many properties (flatness, properness, modular interpretation) as possible. 
Many such extensions have been found, 
see \cite{Cap94, Est01, Esteves_2016,fpv, fpv2, KP19, Mel16, melo,Melo2019Universal,  PT23, Pan96,   Sim94,  viviani}.
In our main results, we will discuss the family of \emph{smoothable fine compactified Jacobians} defined by Pagani--Tommasi \cite{PT23}. These compactified Jacobians are highly non-unique, for fixed $(g,n,d)$, and depend on the choice of a \emph{stability condition} $\sigma$. We write $$ \overline{\mathcal{J}}_{g,n}^d(\sigma) \rightarrow \overline{\mathcal{M}}_{g,n}\, ,$$ for the compactified Jacobian associated to the stability condition $\sigma$; it is a proper nonsingular Deligne--Mumford stack. 
By \cite[Corollary 6.19]{KP19}, even for $g=1$, there are examples $\Jbar_{1,n}^d(\sigma)$, $\Jbar_{1,n}^d(\sigma')$ of compactified Jacobians for different stability conditions $\sigma, \sigma'$ of equal degree which are not isomorphic as Deligne--Mumford stacks.

Stability conditions $\sigma$ as above do not always exist. When $n=0$, a fine compactified Jacobian extending $\Jcal_g^d \to \mathcal M_g$ exists if and only if $\gcd(g-1+d,2g-2)=1$. When $n>0$, fine compactified Jacobians exist for all $g$ and $d$, but it is not always the case that an $S_n$-equivariant stability condition exists \cite[Remark 7.7]{fpv}.

In light of the discussion in \cref{intro: uncompactified}, it is natural to ask to what extent the cohomology of $\overline{\mathcal{J}}_{g,n}^d(\sigma)$ depends on $d$ and on $\sigma$. The surprising answer is: \emph{not at all!}
\begin{theorem}\label{maintheorem}
    For all $2g-2+n>0$, all $d,d' \in \Z$, and all choices of Pagani--Tommasi stability conditions $\sigma$ and $\sigma'$, there exists an isomorphism
    \[ H^*(\overline{\mathcal J}_{g,n}^{d}(\sigma),\Q) \cong H^*(\overline{\mathcal J}_{g,n}^{d'}(\sigma'),\Q) \]of $\Q$-Hodge structures. If the stability conditions $\sigma$ and $\sigma'$ are both invariant for a subgroup $S_\lambda = S_{n_1}\times \dots \times S_{n_p}$, with $n_1+\dots+n_p=n$, then the isomorphism may be chosen to be equivariant with respect to the natural actions of the group $S_\lambda$. 
\end{theorem}

\cref{maintheorem} is a direct consequence of a much more general theorem of Migliorini--Shende--Viviani \cite{MSV21}. Indeed, consider a compactified Jacobian $f\colon \overline{\mathcal J}_{g,n}^{d}(\sigma) \to \overline{\mathcal{M}}_{g,n}$. To prove \cref{maintheorem}, it suffices to prove that $Rf_*\Q$ is independent of $d$ and $\sigma$. By \cite[Theorem 1.7]{MSV21}, all summands in the Decomposition Theorem for $Rf_*\Q$ have full support. Therefore, $Rf_*\Q$ is the intermediate extension of its restriction to $\mathcal M_{g,n}$, so it suffices to prove that $Rf_*\Q\vert_{\mathcal{M}_{g,n}}$ is independent of $d$ and $\sigma$ as an object of the derived category of $S_n$-equivariant mixed Hodge modules. Independence of $\sigma$ is obvious, and independence of $d$ is proven here in \cref{rigidity section}  and constitutes the core of the proof of \cref{uncompactified theorem}. See also \cite[Section 2]{BaeMSY} and \cite[Corollary B]{fpv}.


From the point of view of the boundary combinatorics of these moduli spaces, \cref{maintheorem} is rather surprising. The space $\overline{\mathcal J}_{g,n}^{d}(\sigma)$ admits a stratification, in which each stratum is a torus bundle over a product of smaller moduli spaces $\mathcal J_{g_i,n_i}^{d_i}$, modulo a finite group. One could imagine calculating (say) the Hodge--Deligne polynomial for the space $\overline{\mathcal J}_{g,n}^{d}(\sigma)$ by adding up the Hodge--Deligne polynomials for all of the individual strata.  \cref{uncompactified theorem} implies that the Hodge--Deligne  polynomial of every such stratum is independent of the tuple of degrees $(d_i)$. But this does not immediately prove \cref{maintheorem}, because the sets of strata for different stability conditions $\sigma$ and $\sigma'$ are  different! \cref{maintheorem} thus predicts that there should be some kind of bijection between strata that makes the end result independent of the stability condition. 

A Pagani--Tommasi stability condition is essentially a purely combinatorial object. It is natural to ask for a direct combinatorial proof of \cref{maintheorem}, by explicitly adding up contributions from all strata and arguing that the terms in these sums for different stability conditions $\sigma$, $\sigma'$, can be paired in such a way that the end result is manifestly independent of the stability condition. The main goal of this paper is to give such a proof.

\begin{remark}
    A philosophically related invariance result for virtual Euler characteristics was proven by the fourth named author \cite{wood24}: 
    the orbifold Euler characteristic of the Pagani--Tommasi compactified
universal Jacobian is
$$\chi_{\mathrm{orb}}(  \overline{\mathcal{J}}_{g,n}^d(\sigma) ) = 
\frac{1}{{2^g}\cdot  g!} \, \chi(\overline{\mathcal{M}}_{0,2g+n})\, , $$
which is in particular independent of $d$ and $\sigma$.
\end{remark}

By adding up contributions of strata, the Hodge--Deligne polynomial of $\overline{\mathcal{J}}_{g,n}^d(\sigma)$
can be computed if 
the equivariant Euler characteristics $\chi_c^{S_{n'}}(\Jcal_{g',n'})$ of the uncompactified moduli spaces, with 
$$g' \leq g \ \ \text{and}\ \  2g' - 2 + n' \leq 2g-2+n\, ,$$
are known.
	At the moment, full knowledge of $\chi_c^{S_{n'}}(\Jcal_{g',n'})$ is
    only available for $g'\leq 2$. A computational discussion can be found in Section \ref{compute}.

\subsection{Further invariance properties} \label{intro: further}
One can ask more generally which other cohomological structures of $\overline{\mathcal{J}}_{g,n}^d(\sigma)$ are independent of $d$ and $\sigma\,$? A first natural question is whether the isomorphism in \cref{maintheorem} can be made compatible with the \emph{cup-product} on both sides. This question has been answered in the {\em negative} by 
Bae--Maulik--Shen--Yin \cite{BaeMSY}: the cohomology rings
$H^*(\overline{\mathcal{J}}_{g,n}^d(\sigma))$ depend upon 
$\sigma$ and $d$, in general.

An {\em intrinsic} cohomology ring $\mathbb{H}^*(\overline{\mathcal{J}}_{g,n})$
is defined by
Bae--Maulik--Shen--Yin \cite{BaeMSY}
by taking the associated graded algebra of
$H^*(\overline{\mathcal{J}}_{g,n}^d(\sigma))$
with respect to the perverse filtration
and is
proven  to be independent of
$\sigma$ and $d$. 
In an upcoming paper of 
Bae--Pixton \cite{BaePixtonCyclesUniversalJacobians}, the intrinsic
cohomology ring $\mathbb{H}^*(\overline{\mathcal{J}}_{g,n})$ is shown to be isomorphic to  an explicit subquotient of $\bigoplus_{m \geq 0} H^*(\Mbar_{g,n+m})^{S_m}$ equipped with the cross product.
These developments use several techniques: the Decomposition Theorem \cite{MR751966, MR2525735, MR2653248}, the Fourier--Mukai transformation \cite{arin,  MR2244106, MR4019895,MR607081} and the 
double ramification cycle \cite{BHPSS, HMPPS,JPPZ}. 

Even though the cup-product is not in general invariant, there is still a possibility of exploring integrals (or, more generally, push-forwards to $\overline{\mathcal{M}}_{g,n}$) of tautological classes{\footnote{See \cite{bae2025fourier} for a definition of the tautological ring  $RH^*(\overline{\mathcal{J}}_{g,n}^d(\sigma))\subset H^*(\overline{\mathcal{J}}_{g,n}^d(\sigma))$.}} over 
$\overline{\mathcal{J}}_{g,n}^d(\sigma)$ which are independent of $\sigma$ and $d$.
A beautiful set of such push-forwards
has been found by Bae--Molcho--Pixton, see \cite[Theorem 1.1]{bae2025fourier}.






A different direction concerns the bounded derived category of coherent sheaves on Deligne--Mumford stacks: 
\begin{question}Are the moduli spaces
$\overline{\mathcal{J}}_{g,n}^d(\sigma)$ and  $\overline{\mathcal{J}}_{g,n}^{d'}(\sigma')$
derived equivalent?
    \end{question}

The answer to Question A is affirmative.{\footnote{\label{foot:discussions}Discussions with Y.~Bae, D.~Maulik, J.~Shen, and F.~Viviani in Les Diablerets in May 2026 significantly improved
our understanding of the issues here (following
Melo--Rapagnetta--Viviani \cite{MR4019895} and Arinkin
\cite{arin}).}}
If $n\geq 1$ and the stability conditions are Kass--Pagani, the derived equivalence is proven in 
\cite[Theorem 5.1]{bae2025fourier}. In full generality (with $n\geq 0$ and with Pagani--Tommasi stability conditions),
the derived equivalence is \cite[Corollary B]{fpv}. 

The invariance of the additive structure of cohomology proven in \cref{maintheorem} can also be considered  in   Chen--Ruan orbifold
cohomology. The derived equivalence of Question A implies that the orbifold $K$-theories of 
$\overline{\mathcal{J}}_{g,n}^d(\sigma)$ and  $\overline{\mathcal{J}}_{g,n}^{d'}(\sigma')$ are additively isomorphic.
As a consequence, the Chen--Ruan orbifold cohomologies of  $\overline{\mathcal{J}}_{g,n}^d(\sigma)$ and  $\overline{\mathcal{J}}_{g,n}^{d'}(\sigma')$ are isomorphic as $\mathbb{Q}$-vector spaces. The following question concerning the Chen--Ruan grading is open:

\begin{question}Are the Chen--Ruan orbifold cohomologies of  $\overline{\mathcal{J}}_{g,n}^d(\sigma)$ and  $\overline{\mathcal{J}}_{g,n}^{d'}(\sigma')$ isomorphic as graded $\mathbb{Q}$-vector spaces?
    \end{question}

Stability conditions 
can also be used to find proper and nonsingular  extensions, 
$$\pi: \overline{\mathcal{U}}_{g,n}^{r,d}(\sigma) \to \overline{\mathcal{M}}_{g,n}\, ,$$
of the
universal moduli spaces of stable bundles
of rank $r$ and degree $d$,
$$\pi: \mathcal{U}_{g,n}^{r,d} \to \mathcal{M}_{g,n}\, ,$$
see \cite{Fringuelli} for a study of such extensions.
Given that the Betti numbers of the moduli space
of stable bundles $\mathcal{U}^{r,d}_C$ over a fixed nonsingular
curve $C$ in general {\em depend upon the degree $d$} \cite{MR702806,DesaleRam,  MR1806453}, the first reasonable question to 
ask in higher rank is for fixed $d$.

\begin{question}
    Are the Hodge numbers of  $\, \overline{\mathcal{U}}_{g,n}^{r,d}(\sigma)$
independent of the choice of stability condition~$\sigma$?

\end{question}

Note: the discussion and future questions in Section \ref{intro: further} were
sharpened based on communication with several
people (see footnote \ref{foot:discussions}).

\subsection*{Acknowledgements}
We thank
Younghan Bae,
Samir Canning,
Lycka Drakengren,
Aitor Iribar L\'opez, Davesh Maulik, Miguel Moreira, Aaron Pixton,
Junliang Shen, 
and Dimitri
Zvonkine for discussions about
the geometry of compactified universal Jacobians. The project began
during a visit to the
{\em Mathematisches Forschungsinstitut
Oberwolfach} in the summer of 2024. 
The results were presented in the {\em Algebraic Geometry Seminar} at Humboldt University in Berlin \cite{PandharipandeGeometryUniversalJacobian2025} in February 2025.
Discussions
with Sam Molcho and Filippo Viviani during the conference {\em Harmonies in Moduli} in Rome in the summer of 2025 clarified several issues about the log structures of the compactifications. Marco Fava, Nicola Pagani, and Filippo Viviani proved  \cite[Remark 7.7]{fpv} in response to a question of ours, which allowed us to  sharpen our results here. We are grateful to Jeremy Feusi and Qizheng Yin for contributing the arguments presented in the Appendix. We thank the referee for several useful comments.

R.P.\ and J.S.\ were supported
by SNF-200020-219369 and SwissMAP. D.P.\ was supported by the National Science Foundation
under Grant No.~DMS-2424441, and a Wallenberg Scholar fellowship.

\section{Cohomology of universal Jacobians}\label{rigidity section}


\subsection{Rigidity} 
	
	We will use the following Rigidity Theorem for variations of Hodge structure proven by Schmid \cite[(7.24)]{schmid}, extending previous results of Grothendieck, Griffiths, and Deligne. All schemes and stacks here are separated and of finite type over $\C$.

	\begin{theorem}[Rigidity Theorem]Let $B$ be a smooth connected algebraic space, $x \in B$ a point. Suppose $V$ and $W$ are polarizable variations of Hodge structure over $B$. Suppose given an isomorphism $\phi: V_x \to W_x$ of Hodge structures with an action of $\pi_1(B,x)$. Then there is a unique isomorphism $V\to W$ of variations of Hodge structure extending $\phi$. 
		\end{theorem}

	We will need an extension of the Rigidity Theorem to stacks. The following will be sufficient for our purposes.
	
	\begin{proposition}\label{rigidity for stacks}Suppose that $B$ is a smooth connected Artin stack, such that there exists an atlas $U \to B$ with connected fibers. Then the conclusion of the Rigidity Theorem holds for $B$. Explicitly, suppose $V$ and $W$ are polarizable variations of Hodge structure over $B$, and let $x$ be a point of $B$. Suppose given an isomorphism $\phi: V_x \to W_x$ of Hodge structures with an action of $\pi_1(B,x)$. Then there is a unique isomorphism $V\to W$ of variations of Hodge structure extending $\phi$. \end{proposition}
	
	\begin{proof}Let $U^k$ be the $k$-fold fibered product $U \times_B U \times_B \dots \times_B U$. Let $y \in U$ be a preimage of $x$. From the choice of $y$ we get a compatible system of basepoints on all powers $U^k$. Let $f_k:U^k \to B$ be the projection. The fiber at the basepoint of $f^\ast_kV$ is canonically identified with $V_x$, for all $k$, and similarly for $W$. Since each $U^k$ is smooth and connected, the Rigidity Theorem furnishes a compatible system of isomorphisms $f^\ast_k V \to f^\ast_k W$ for all $k \geq 1$. These isomorphisms descend uniquely to an isomorphism $V \to W$. 
	\end{proof}
	
	\begin{corollary}\label{lem:rigidity}
	    Let $f:X\to Y$ be a morphism of Artin stacks, with $Y$ satisfying the conditions of \cref{rigidity for stacks}. Assume that $\pi_1(X)\to\pi_1(Y)$ is surjective. If $V$ and $W$ are variations of Hodge structure on $Y$ such that $f^\ast(V)$ is isomorphic to $f^\ast(W)$, then $V$ and $W$ are isomorphic.
	\end{corollary} 

    \begin{proof}
        By \cref{rigidity for stacks}, it suffices to find a point $y \in Y$ such that there exists an isomorphism $V_y \to W_y$ of Hodge structures equipped with an action of $\pi_1(Y)$. Now if $x \in X$ then $f^\ast(V)_x \cong V_{f(x)}$ and the action of $\pi_1(X)$ is through the surjection to $\pi_1(Y)$, so the result follows. 
    \end{proof}

\subsection{Universal Jacobians}
Let $\mathcal J_{g,n}^d$ be the degree $d$ universal Jacobian over $\Mcal_{g,n}$, and let  $$\pi^d : \Jcal_{g,n}^d \to \Mcal_{g,n}$$ be the projection (with the degree $d$ now explicitly labelled). 

\begin{proposition}\label{prop:isom of local systems}
    For any integers $d, d' $, and any $g$, $n$ and $q$ satisfying $2g-2+n>0$, there is an isomorphism $$R^q \pi^d_\ast \Z \cong R^q \pi_\ast^{d'} \Z $$ of $S_n$-equivariant polarizable variations of Hodge structures on $\Mcal_{g,n}$.

\end{proposition}

\begin{proof}
	The result is trivial when $g=0$, since in this case all maps $\pi^d$ are isomorphisms, and when $n=1$, since $\Jcal_{g,1}^d \cong \Jcal_{g,1}^{d'}$ for any two $d,d'$. The result for $n>1$ follows from the result for $n=0$, since both sheaves $\smash{R^q \pi^d_\ast \Z}$ and $\smash{R^q \pi_\ast^{d'} \Z}$ are pulled back from $\Mcal_g$, and $S_n$ acts fiberwise. (Indeed, the family $\Jcal_{g,n}^d \to \Mcal_{g,n}$ is pulled back from $\Mcal_g$ for all $d$ and $n$.) 
	But the result for $n=0$ follows from the result for $n=1$, and Corollary \ref{lem:rigidity}. Indeed: $\pi_1(\Mcal_{g,1})\to \pi_1(\Mcal_g)$ is surjective, since a fiber bundle with connected fibers induces a surjection on fundamental groups, using the long exact sequence in homotopy associated to a fibration. Moreover, $\Mcal_g$ admits an atlas with connected fibers: if $g>0$, we may take the atlas to be the space $\Mcal_g^1$ parametrizing curves with a marked point and a nonzero tangent vector at that point; a curve of positive genus cannot have an automorphism fixing a point and a nonzero tangent vector in characteristic zero.
\end{proof}

\begin{proposition}\label{cohomology invariance on interior}
	For any integers $d,d'$ there are isomorphisms $$H^\ast(\Jcal_{g,n}^d,\Q)\cong H^\ast(\Jcal_{g,n}^{d'},\Q)\quad \text{and}\quad  H^\ast_c(\Jcal_{g,n}^d,\Q)\cong H^\ast_c(\Jcal_{g,n}^{d'},\Q)\, ,$$ both compatible with mixed Hodge structure, the action of $S_n$, and the cup-product. 
\end{proposition}
\begin{proof}Suppose that $p:X\to Y$ is a smooth proper family. There is an isomorphism\begin{equation}
  R p_\ast \Q \cong \bigoplus_q R^q p_\ast \Q [-q]\,.  \label{deligne degeneration}
\end{equation}
	The fact that there is such an isomorphism in the derived category of constructible sheaves is Deligne's degeneration theorem \cite{delignedegeneration}. Saito \cite{saitomixedhodge} promoted this result to an isomorphism in his derived category of mixed Hodge modules. It follows that there are isomorphisms of mixed Hodge structures
	\begin{equation}\label{eq: two isomorphisms} H^k(X) \cong \bigoplus_{p+q=k} H^p(Y,R^q p_\ast\Q) \quad \text{and}\quad H^k_c(X) \cong \bigoplus_{p+q=k} H^p_c(Y,R^q p_\ast\Q)\,. \end{equation}
	(When $Y$ is smooth, these isomorphisms are interchanged under Poincar\'e duality.) 	
	Now suppose $p$ is an abelian scheme. Then the isomorphism \eqref{deligne degeneration} can be chosen to be compatible with cup-product, so that both isomorphisms \eqref{eq: two isomorphisms} are isomorphisms of rings. (This is a very rare property.) In fact, for an abelian scheme, the natural map
	\begin{equation}\label{subtle map}
	    \bigwedge R^1 p_\ast \Q[-1] \longrightarrow Rp_\ast\Q
	\end{equation} 
	induced by cup-product is an isomorphism; indeed, its mapping cone is acyclic. The same holds for torsors under abelian schemes, by the same argument. 
	
	After the above considerations, it suffices to prove that for the various projections $p=\pi^d : \Jcal_{g,n}^d \to \Mcal_{g,n}$ we have isomorphisms  $\smash{R^1 \pi^d_\ast \Q \cong R^1 \pi_\ast^{d'} \Q}$. But this was proven in \cref{prop:isom of local systems}.\end{proof}
	
	\begin{remark}The  multiplicative isomorphism \eqref{deligne degeneration} can also be realized from the more general motivic decomposition proven by Deninger and Murre \cite{deningermurre}. The theorem of Deninger and Murre was generalized to torsors for abelian schemes in \cite[Appendix A]{BaeMSY}.		
	\end{remark}

    \begin{remark}
        The construction of the morphism \eqref{subtle map} is somewhat subtle. The left-hand side is the free commutative algebra generated by $R^1 p_*\Q[-1]$. Since $Rp_\ast\Q$ is a commutative algebra object, the universal property of $\bigwedge$ says that the linear map  $R^1 p_*\Q[-1] \to Rp_*\Q$ in the derived category is adjoint to a ring homomorphism $\bigwedge R^1 p_*\Q[-1] \to Rp_*\Q $, and this is how \eqref{subtle map} is defined. But in order to speak of $Rp_*\Q$ as being a commutative algebra object, or of $\bigwedge$ having a universal property, we are obliged to treat the derived category as a stable $\infty$-category \cite{HigherAlgebra}. An enhancement of the derived category of mixed Hodge modules to a stable $\infty$-category compatibly with six operations was recently constructed by Tubach \cite{tubach}. 
    \end{remark}


\section{Boundary stratifications} \label{sect:boundary_strat}

\subsection{Hodge Euler characteristics}\label{sec: hodge euler characteristics}

Let $\mathsf{MHS}$ denote the category of $\Q$-mixed Hodge structures, always assumed to be graded-polarizable, and let $K_0(\mathsf{MHS})$ denote its Grothendieck group.

\begin{defin}
    Let $V=V^\ast$ be a bounded cohomologically graded object in $\mathsf{MHS}$. By the \emph{class of $V$ in $K_0(\mathsf{MHS})$} we mean the alternating sum
    \[ \sum_i (-1)^i [V^i].\]
    Similarly, if $V^{\ast,\ast}$ has a cohomological bigrading, then the class of $V$ is $\sum_{p,q} (-1)^{p+q}[V^{p,q}]$.
\end{defin}

\begin{defin}
    Let $X$ be a finite type Deligne--Mumford stack over $\C$. We define the \emph{Hodge Euler characteristic} $\chi_c(X)$ to be class of $H^\ast_c(X,\Q)$ in $K_0(\mathsf{MHS})$.
\end{defin}

\begin{lemma}\label{l1}
    Suppose that $X$ and $Y$ are both smooth and proper. Then $\chi_c(X) = \chi_c(Y)$ if and only if $H^i(X,\Q) \cong H^i(Y,\Q)$ as Hodge structures for all $i$.
\end{lemma}

\begin{proof}By purity of weights, the class of $H^i(X,\Q)=H^i_c(X,\Q)$ in $K_0(\mathsf{MHS})$ can be recovered from $\chi_c(X)$ by taking all terms of weight $i$. But this determines the Hodge structure $H^i(X,\Q)$ completely, since the category of pure polarizable Hodge structures is semisimple. 
\end{proof}

\begin{lemma}\label{l2}
    Let $X = \bigcup X_{\alpha \in A}$ be a stratification into pairwise disjoint locally closed substacks. Then $\chi_c(X) = \sum_{\alpha \in A} \chi_c(X_\alpha)$.
\end{lemma}

\begin{proof}The proof reduces by induction to the case $|A|=2$, in which case one uses the long exact sequence
\[ \dots \to H^i_c(U,\Q) \to H^i_c(X,\Q)\to H^i_c(X\setminus U,\Q)\to H^{i+1}_c(U,\Q)\to \dots \]
associated to an open immersion $U \subset X$.   
\end{proof}

Suppose that a finite group $G$ acts on $X$. Then $\chi_c(X)$ admits an equivariant refinement $\chi_c^{G}(X)$, valued in $K_0(\mathsf{MHS}^G)$, where $\mathsf{MHS}^G$ denotes the category of $\Q$-mixed Hodge structures with an action of $G$. The evident analogue of \cref{l1} is true in the equivariant setting, with the same argument. Moreover, the same argument as \cref{l2} implies that $\chi_c^G(-)$ is additive over $G$-invariant stratifications. 


The following lemma gives a generalization in the case where $G$ is allowed to permute the strata.

	\begin{lemma}\label{equivariant stratification} Let $X=\bigcup_{\alpha \in A} X_\alpha$ be a stratification into pairwise disjoint locally closed substacks. Suppose given a finite group $G$ acting on $X$, in such a way that for all $g \in G$ and $\alpha \in A$, $g X_\alpha = X_\beta$ for some $\beta \in A$. Then\[ \chi_c^G(X) = \sum_{\alpha \in A/G} \operatorname{Ind}_{\mathrm{Stab}(\alpha)}^G \chi_c^{\mathrm{Stab}(\alpha)}(X_\alpha),\]
		where the sum runs over any set of orbit representatives of $G$ acting on $A$. 
	\end{lemma}

\begin{proof}
	Since $\chi_c^G(-)$ is additive over $G$-invariant stratifications, we may treat one $G$-orbit of strata at a time, and assume that $G$ acts transitively on $A$. Then necessarily $X = \coprod_{\alpha \in A} X_\alpha$ is actually a disjoint union, and 
	\[ H^*_c(X,\Q) = \bigoplus_{\alpha \in A} H^*_c(X_\alpha,\Q) \cong  \operatorname{Ind}_{\mathrm{Stab}(\alpha)}^G H^*_c(X_\alpha,\Q)\]
	for any choice of $\alpha \in A$.
\end{proof}

 The above discussion provides a first reduction of \cref{maintheorem}. Firstly, it suffices by \cref{l1} to prove that the Hodge Euler characteristic $\chi_c^{S_\lambda}(\Jbar_{g,n}^d(\sigma))$ is independent of $d$ and $\sigma$. Secondly,  by \cref{equivariant stratification} we may compute $\chi_c^{S_\lambda}(\Jbar_{g,n}^d(\sigma))$ by adding up contributions from individual strata. In analyzing the individual strata, we use that 
 \begin{equation}\label{invariance}
     \chi_c^{S_n}(\mathcal J_{g,n}^d) = \chi_c^{S_n}(\mathcal J_{g,n}^{d'})
 \end{equation}
 for any $d$ and $d'$,  by \cref{cohomology invariance on interior}.

\subsection{Stratification of the universal stack}\label{sec:stratification of universal stack}



Recall that a prestable curve is said to be \emph{quasi-stable} if all its unstable components are disjoint, and that a line bundle on a quasi-stable curve is \emph{admissible} if its degree on each unstable component is $1$. If $C$ is a quasi-stable curve and $\mathcal L$ is an admissible line bundle on $C$, then $|\mathrm{Aut}(C,\mathcal L)/\mathbb G_m|<\infty$ if and only if the union of all stable components of $C$ is connected, and we shall assume this in what follows. Let $\mathfrak J_{g,n}$ be the highly nonseparated Deligne--Mumford stack parametrizing such pairs $(C,\mathcal L)$. Stabilization defines a natural morphism $\mathfrak J_{g,n} \to \Mbar_{g,n}$. 

In this section we will describe a stratification of $\mathfrak J_{g,n}$, analogous to the stratification of $\Mbar_{g,n}$ by topological type. In these terms, a stability condition $\sigma$ will then simply be the data of a finite collection of strata in $\mathfrak{J}_{g,n}$, satisfying a short list of combinatorial conditions which ensure that this union of strata defines an open substack $$\Jbar_{g,n}^d(\sigma) \subset \mathfrak{J}_{g,n}$$ which is proper over $\Mbar_{g,n}$. The stratification of $\mathfrak{J}_{g,n}$ is indexed by triples 
\begin{equation}
    \label{triple} (\Gamma,\Gamma_0,\underline d)
\end{equation}
where $\Gamma$ is a stable graph of type $(g,n)$, $\Gamma_0 \subseteq \Gamma$ is a {connected} subgraph containing all vertices, and $\underline d : V(\Gamma) \to \Z$ is a function. Geometrically, this stratum parametrizes quasi-stable curves whose dual graph is obtained from $\Gamma$ by subdividing each edge \emph{not} in $\Gamma_0$ once, equipped with an admissible line bundle whose degree on each stable component is specified by the function $\underline d$. It will be convenient in what follows to write $e_{\Gamma,\Gamma_0} = |E(\Gamma) \setminus E(\Gamma_0)|$ for the number of unstable components of the quasi-stable curve corresponding to $(\Gamma,\Gamma_0)$.

Let us recall something about the structure of the Jacobian of a nodal curve. Let $f:C\to S$ be a flat proper family of nodal curves, and assume that the dual graph $\gamma$ is constant in the family. Let $\nu: C^\nu \to C$ be its fiberwise normalization, and let $i:C^{\mathrm{sing}}\to C$ be the fiberwise singular locus. The exact sequence $0 \to \mathcal O_C^\times \to \nu_* \mathcal O_{C^\nu}^\times \to i_* \mathcal O_{C^{\mathrm{sing}}}^\times\to 0$ gives an exact sequence on $S$
\[ 0\to f_* \mathcal O_C^\times \to (f\circ \nu)_* \mathcal O^\times_{C^\nu} \to (f\circ i)_*\mathcal O^\times_{C^{\mathrm{sing}}} \to R^1 f_* \mathcal O_C^\times \to R^1 (f\circ \nu)_* \mathcal O^\times_{C^\nu} \to 0,\]
or more explicitly
\begin{equation}\label{exact sequence of etale sheaves} 0 \to \mathcal O_S^\times \to \bigoplus_{v \in V(\gamma)} \mathcal O_S^\times\to \bigoplus_{e \in E(\gamma)} \mathcal O_S^\times \to R^1 f_* \mathcal O_C^\times \to R^1 (f\circ \nu)_* \mathcal O^\times_{C^\nu} \to 0.\end{equation}
On the level of the group schemes representing these \'etale sheaves, this says that $\mathrm{Pic}_{C/S}$ is fibered over $\mathrm{Pic}_{C^\nu/S}$, with kernel the algebraic torus $\mathrm{Hom}(H_1(\gamma,\Z),\mathbb G_m)$, which represents the cokernel of the map $\bigoplus_{v \in V(\gamma)} \mathcal O_S^\times\to \bigoplus_{e \in E(\gamma)}\mathcal O_S^\times$.

We are now ready to describe the stratification of $\mathfrak{J}_{g,n}$ by topological type. Let $(\Gamma,\Gamma_0,\underline d)$ be a triple as in \eqref{triple}. Let $\mathrm{Aut}(\Gamma,\Gamma_0,\underline d)$ denote the subgroup of $\mathrm{Aut}(\Gamma)$ of graph automorphisms which send the subgraph $\Gamma_0$ to itself, and which preserve the function $\underline d: V(\Gamma)\to \Z$. We then have the following, which is the analogue of the familiar stratification
\[ \Mbar_{g,n} = \bigcup_{[\Gamma]} [\mathcal M_\Gamma/\mathrm{Aut}(\Gamma)]\]
of the Deligne--Mumford compactification. 

\begin{proposition}\label{stratification of universal stack}
    For each $(\Gamma,\Gamma_0,\underline d)$ as in \eqref{triple}, there is a torus bundle 
    \[ \Jcal_{\Gamma,\Gamma_0,\underline d} \longrightarrow \prod_{v \in V(\Gamma)} \Jcal_{g(v),n(v)}^{\underline d(v)}\]
    which is an $\mathrm{Aut}(\Gamma,\Gamma_0,\underline d)$-equivariant torsor for the torus $T_{\Gamma_0} := \mathrm{Hom}(H_1(\Gamma_0),\mathbb G_m) \cong \mathbb G_m^{b_1(\Gamma_0)}$. 
    The stratum in $\mathfrak J_{g,n}$ corresponding to $(\Gamma,\Gamma_0,\underline d)$ is isomorphic to $[\Jcal_{\Gamma,\Gamma_0,\underline d}/\mathrm{Aut}(\Gamma,\Gamma_0,\underline d)]$, so that set-theoretically there is a disjoint union
    \[ \mathfrak J_{g,n} = \bigcup_{[\Gamma,\Gamma_0,\underline d]} [\Jcal_{\Gamma,\Gamma_0,\underline d}/\mathrm{Aut}(\Gamma,\Gamma_0,\underline d)].\]
\end{proposition}

\begin{proof}
    	Let $\Gamma^{\mathrm{qs}}$ be the graph obtained from $\Gamma$ by subdividing each edge not in $\Gamma_0$ once. Set 
		\[ \ \mathcal M_{\Gamma^{\mathrm{qs}}} = \prod_{v \in V(\Gamma^{\mathrm{qs}})} \mathcal M_{g(v),n(v)} \cong \mathcal M_\Gamma \times (B\mathbb G_m)^{e_{\Gamma,\Gamma_0}}.\]
		Then  $\mathcal M_{\Gamma^{\mathrm{qs}}}$ is the moduli space of quasi-stable  curves equipped with an isomorphism of their dual graph with $\Gamma^{\mathrm{qs}}$. We extend the function $\underline d$ to $V(\Gamma^{\mathrm{qs}})$ by setting $\underline d(v)=1$ for all $v \in V(\Gamma^{\mathrm{qs}})\setminus V(\Gamma)$. Let $\mathcal C \to \mathcal M_{\Gamma^{\mathrm{qs}}}$ be the universal quasi-stable curve. We let $\Jcal_{\Gamma,\Gamma_0,\underline d} $ be the component of the relative Jacobian of $\mathcal C$ over $\mathcal M_{\Gamma^{\mathrm{qs}}}$, parametrizing admissible line bundles of multidegree $\underline d$. 	There are now natural maps
		\[ \Jcal_{\Gamma,\Gamma_0,\underline d} \stackrel f \longrightarrow \prod_{v \in V(\Gamma^{\mathrm{qs}})} \Jcal_{g(v),n(v)}^{\underline d(v)} \stackrel g \longrightarrow \prod_{v \in V(\Gamma)} \Jcal_{g(v),n(v)}^{\underline d(v)} \]
		given by restricting a line bundle on the quasi-stable curve to its normalization, and forgetting the unstable components, respectively. The map $f$ is a torus bundle with fiber $\mathrm{Hom}(H_1(\Gamma^{\mathrm{qs}},\Z),\mathbb G_m)$, as we saw from \eqref{exact sequence of etale sheaves}. The map $g$ is a trivial bundle with fiber $(B\mathbb G_m)^{e_{\Gamma,\Gamma_0}}$, since $\mathcal J_{0,2}^d \cong\mathcal M_{0,2} \cong B\mathbb G_m$ for all $d$.  Consequently, the fiber of $(g\circ f)$ is the quotient torus
		\[ \mathrm{Hom}(H_1(\Gamma^{\mathrm{qs}},\Z),\mathbb G_m)/\mathbb G_m^{e_{\Gamma,\Gamma_0}}.  \]
		The naturality of the sequence \eqref{exact sequence of etale sheaves} shows that $\mathbb G_m^{e_{\Gamma,\Gamma_0}}$ acts in the evident way on $\bigoplus_{v \in V(\Gamma^{\mathrm{qs}})} \mathbb G_m$ and $\bigoplus_{e \in E(\Gamma^{\mathrm{qs}})}\mathbb G_m$, which allows us to identify the quotient torus with $\mathrm{Hom}(H_1(\Gamma_0,\Z),\mathbb G_m)$. Since $\mathcal M_{\Gamma^{\mathrm{qs}}}$ parametrizes quasi-stable curves together with an identification of their dual graph with $\Gamma^{\mathrm{qs}}$, the corresponding locus in the moduli stack of prestable curves is obtained by forgetting this identification, giving the quotient stack $[\mathcal M_{\Gamma^{\mathrm{qs}}}/\mathrm{Aut}(\Gamma,\Gamma_0)]$. Similarly the locus in $\mathfrak J_{g,n}$ corresponding to $\Jcal_{\Gamma,\Gamma_0,\underline d}$ is $[\Jcal_{\Gamma,\Gamma_0,\underline d} /\mathrm{Aut}(\Gamma,\Gamma_0,\underline d)]$. \end{proof}

\begin{remark}Let us make explicit the notion of ``equivariant torsor'' which appears in the statement of \cref{stratification of universal stack}. Suppose $G$ and $H$ are groups, that $G$ acts on a space $X$, and $G$ acts on the group $H$. A \emph{$G$-equivariant $H$-torsor over $X$} is a lifting in the following diagram:
	\[ \begin{tikzcd}
		& B(H \rtimes G) \arrow[d]\\
		\left[X/G\right] \arrow[r]\arrow[ur,dashed]& BG.
	\end{tikzcd}\]
In particular, a $G$-equivariant $H$-torsor has an underlying $H$-torsor $E \to X$, with a $G$-action on $E$ making the projection equivariant, such that the actions of $G$ and $H$ on $E$ combine to an action of the semidirect product $H \rtimes G$. If $G$ acts trivially on $H$, this reduces to the notion of an $H$-torsor on the quotient stack $[X/G]$. 
\end{remark}

\subsection{Pagani--Tommasi stability conditions}

Given a stable graph $\Gamma$, we denote by
\begin{equation}
    \Pic(\Gamma) = \mathbb{Z}^{V(\Gamma)}/\mathrm{Twist}_\Gamma
\end{equation}
the \emph{Picard group of $\Gamma$}. Here $\mathrm{Twist}_\Gamma$ is the group of multidegrees obtained by twisting by some vertex of $\Gamma$, see \cite{MR2355607} for further details. We have that $\Pic(\Gamma)=\coprod_{d \in \mathbb{Z}} \Pic^d(\Gamma)$ decomposes into fibers according to the total degree of a divisor on $\Gamma$ (i.e. the quantity $d = \sum_{v \in V(\Gamma)} \underline{d}(v)$), and each $\Pic^d(\Gamma)$ is a torsor under the finite abelian group $\Pic^0(\Gamma)$.

We will not need to use the full definition of a {Pagani--Tommasi stability condition}. The only thing we shall need to know is that a Pagani--Tommasi stability condition $\sigma$ over $\Mbar_{g,n}$ of degree $d$ consists of a collection of triples $(\Gamma,\Gamma_0,\underline d)$ as in \eqref{triple}, which in particular must satisfy the following conditions:
\begin{enumerate}[(i)]
	\item \label{condition1} For every pair $(\Gamma,\Gamma_0)$, the set 
	\[\{ \underline d  : (\Gamma,\Gamma_0,\underline d) \in \sigma\}\]
	is a minimal set of coset representatives for $\operatorname{Pic}^{d-e_{\Gamma,\Gamma_0}}(\Gamma_0)$. 
	\item \label{equivariant} If $f: \Gamma' \to \Gamma$ is an isomorphism of stable graphs, then 
	\[ (\Gamma,\Gamma_0,\underline d) \in \sigma \iff (\Gamma',f^{-1}\Gamma_0,\underline d \circ f) \in \sigma.\]
\end{enumerate}

The full definition of a Pagani--Tommasi stability condition is not much more complicated than this: it involves also a compatibility condition under the operation of growing $\Gamma_0$ (while keeping $\Gamma$ fixed), and a compatibility condition under the operation of contracting a subgraph of $\Gamma$. These conditions will however play no role in our arguments. For such a Pagani--Tommasi stability condition $\sigma$ we have 
\begin{equation}\label{pagani tommasi stratification}
    \Jbar_{g,n}^d(\sigma) = \bigcup_{[\Gamma,\Gamma_0,\underline d]\in \sigma} [\Jcal_{\Gamma,\Gamma_0,\underline d} /\mathrm{Aut}(\Gamma,\Gamma_0,\underline d)] \subset \mathfrak J_{g,n}.
\end{equation} 



We claim that at this point, the main result of the paper has been reduced to the following combinatorial claim. 

\begin{lemma}
    \label{combinatorial claim}Let $\sigma$ and $\sigma'$ be Pagani--Tommasi stability conditions of degrees $d$ and $d'$, for fine compactified universal Jacobians over $\Mbar_{g,n}$. Assume moreover that $\sigma$ and $\sigma'$ are invariant under a subgroup $S_\lambda = S_{n_1} \times \dots \times S_{n_p} \subseteq S_n$. Let $\Gamma$ be a stable graph of type $(g,n)$ with $n_k$ legs labeled by $k$, for all $k=1,\dots,p$ (so some legs may be indistinguishable), and let $\Gamma_0 \subset \Gamma$ be a connected  subgraph containing all vertices. There exists an $\mathrm{Aut}(\Gamma,\Gamma_0)$-equivariant bijection between the sets
    \[ \{\underline d : (\Gamma,\Gamma_0,\underline d) \in \sigma\} \qquad \text{and} \qquad \{\underline d' : (\Gamma,\Gamma_0,\underline d') \in \sigma'\}.\]
    (The assumption that $\sigma$ and $\sigma'$ are $S_\lambda$-invariant implies in particular that $\mathrm{Aut}(\Gamma,\Gamma_0)$ acts on both sets.) 
\end{lemma}

Before proceeding to the proof of \cref{combinatorial claim}, let us explain how it implies \cref{maintheorem}. 

\begin{proof}[Proof of \cref{maintheorem}]Consider first the non-equivariant situation; that is, $S_\lambda$ is trivial. It was observed in Section \ref{sec: hodge euler characteristics} that it suffices to show that $\chi_c(\Jbar_{g,n}^d(\sigma))$ is independent of $d$ and $\sigma$.
    By \eqref{pagani tommasi stratification} and \cref{stratification of universal stack} we have
    \[ \chi_c(\Jbar_{g,n}^d(\sigma)) = \sum_{[\Gamma,\Gamma_0,\underline d] \in \sigma} \chi_c\big( \Jcal_{\Gamma,\Gamma_0,\underline d} / \mathrm{Aut}(\Gamma,\Gamma_0,\underline d) \big).\]
    Consider our fixed Pagani--Tommasi stability conditions $\sigma$ and $\sigma'$. For each pair $(\Gamma,\Gamma_0)$, let $\phi : \underline d \mapsto \underline d'$ be the bijection of \cref{combinatorial claim}. It suffices to show that 
    \begin{equation}\label{conclusion}
        \chi_c\big( \Jcal_{\Gamma,\Gamma_0,\underline d} / \mathrm{Aut}(\Gamma,\Gamma_0,\underline d) \big) = \chi_c\big( \Jcal_{\Gamma,\Gamma_0,\phi(\underline d)} / \mathrm{Aut}(\Gamma,\Gamma_0,\phi(\underline d)) \big)
    \end{equation} 
   for all $\underline d$ with $(\Gamma,\Gamma_0,\underline d) \in \sigma$. We consider the Leray--Serre spectral sequence of the torus bundles
    \[ E_2^{pq} = H^p_c(\prod_{v \in V(\Gamma)} \Jcal_{g(v),n(v)}^{\underline d(v)},\Q) \otimes H^q_c(T_{\Gamma_0},\Q) \implies H^{p+q}_c(\Jcal_{\Gamma,\Gamma_0,\underline d},\Q)\]
    and
    \[ E_2^{pq} = H^p_c(\prod_{v \in V(\Gamma)} \Jcal_{g(v),n(v)}^{\phi(\underline d)(v)},\Q) \otimes H^q_c(T_{\Gamma_0},\Q) \implies H^{p+q}_c(\Jcal_{\Gamma,\Gamma_0,\phi(\underline d)},\Q).\]
   Note that in both cases, the fundamental group of the base acts trivially on the cohomology of the fiber. This is because the monodromy group acts through the structure group of the torsor, which is the group $T_{\Gamma_0}$ itself, and a connected group acts trivially on its own cohomology.     The $E_2$-pages are isomorphic by \cref{cohomology invariance on interior}. The groups $\mathrm{Aut}(\Gamma,\Gamma_0,\underline d)$ resp.~$\mathrm{Aut}(\Gamma,\Gamma_0,\phi(\underline d))$ act naturally on the $E_2$-pages and on the abutments. The equivariance of $\phi$ with respect to graph automorphisms implies that $\mathrm{Aut}(\Gamma,\Gamma_0,\underline d) \to \mathrm{Aut}(\Gamma,\Gamma_0,\phi(\underline d))$ is an isomorphism, and that the action on the $E_2$-pages is identified. We use these spectral sequences only at the level of Euler characteristics: the class of the abutment in $K_0(\mathsf{MHS}^{\mathrm{Aut}(\Gamma,\Gamma_0,\underline d)})$ equals the class of any page of the spectral sequence. Since taking invariants is exact over $\mathbb Q$, the equality of the $E_2$-classes gives \eqref{conclusion}.

	Let us now consider the case that $S_\lambda$ is not trivial. We apply \cref{equivariant stratification}, and consider the action of $S_\lambda$ on the stratification by topological type. For a triple $(\Gamma,\Gamma_0,\underline d)$ as above describing a stratum in $\Jbar_{g,n}^d(\sigma)$, let us define 
	\[ \mathrm{Aut}^\lambda(\Gamma,\Gamma_0,\underline d)\]
	as the group of automorphisms of the same graph, except some of the legs are instead treated as indistinguishable: we have $n_k$ legs labelled by $k$, for all $k=1,\dots,p$. There is a natural map $\mathrm{Aut}^\lambda(\Gamma,\Gamma_0,\underline d) \to S_\lambda$, and the usual automorphism group $\mathrm{Aut}(\Gamma,\Gamma_0,\underline d)$ is its kernel. There is then  a short exact sequence
	\[ 1 \to {\mathrm{Aut}}(\Gamma,\Gamma_0,\underline d) \to {\mathrm{Aut}}^\lambda(\Gamma,\Gamma_0,\underline d) \to \mathrm{Stab}_{(\Gamma,\Gamma_0,\underline{d})} \to 1,\]
	where $\mathrm{Stab}_{(\Gamma,\Gamma_0,\underline{d})}$ denotes the subgroup of $S_\lambda$ which takes the stratum corresponding to $(\Gamma,\Gamma_0,\underline d)$ to itself. Now \cref{equivariant stratification} tells us that
	\[ \chi_c^{S_\lambda}(\Jbar_{g,n}^d(\sigma)) = \sum_{[\Gamma,\Gamma_0,\underline d] \,/\,S_\lambda} \operatorname{Ind}_{\mathrm{Stab}(\Gamma,\Gamma_0,\underline d)}^{S_\lambda} \chi_c^{\mathrm{Stab}(\Gamma,\Gamma_0,\underline d)} (\Jcal_{\Gamma,\Gamma_0,\underline d} / {\mathrm{Aut}}(\Gamma,\Gamma_0,\underline d))\]
	with the summation running over $S_\lambda$-orbits of strata, or equivalently, isomorphism classes of graphs with some legs indistinguishable. The group ${\mathrm{Stab}(\Gamma,\Gamma_0,\underline d)}$ acts on $\Jcal_{\Gamma,\Gamma_0,\underline d} / {\mathrm{Aut}}(\Gamma,\Gamma_0,\underline d)$ since ${\mathrm{Aut}}^\lambda(\Gamma,\Gamma_0,\underline d)$ acts on $\Jcal_{\Gamma,\Gamma_0,\underline d}$.   At this point one proceeds as in the non-equivariant case, noting that \cref{combinatorial claim} also gives us an identification $\mathrm{Aut}^\lambda(\Gamma,\Gamma_0,\underline d) \to \mathrm{Aut}^\lambda(\Gamma,\Gamma_0,\phi(\underline d))$.	\end{proof}

\subsection{Proof of the combinatorial claim}

We begin with the following lemma.

\begin{lemma} \label{Lem:dmultiplication}
Fix a stable graph $\Gamma$ of genus $g$ with $n_1+\dots+n_p$ legs, where each set of $n_k$ legs are pairwise indistinguishable. Consider the set
\begin{equation} \label{eqn:Dset}
    D = \{d \in \mathbb{Z} : \gcd(d-g+1, 2g-2,n_1,\dots,n_p)=1\} \subseteq \mathbb{Z}\,.
\end{equation}
Then for $d_1, d_2 \in D$ there exists an $\Aut(\Gamma)$-equivariant bijection
\[
\phi: \Pic^{d_1}(\Gamma) \to \Pic^{d_2}(\Gamma)\,.
\]
\end{lemma}
\begin{proof}Let $M= \gcd(2g-2,n_1,\dots,n_p)$.
Introducing the coordinate change $d'=d'(d)=d-g+1$ we observe that\footnote{The move of shifting by $-g+1$ was inspired by \href{https://chatgpt.com/share/699d7008-dc68-8006-b12b-cefc9263ed29}{a conversation with o3-mini high}, though the full strategy proposed by the model did not work as described.}
\[
D' = \{d'(d): d \in D\} = \{d' \in \mathbb{Z}: [d'] \in (\mathbb{Z}/M \mathbb{Z})^\times\}\,.
\]
In our argument below we construct two types of $\Aut(\Gamma)$-equivariant bijections relating degrees $d_1', d_2'$ of the forms
\begin{itemize}
    \item[a)] $d_2' = d_1' + k M$ for arbitrary $k \in \mathbb{Z}$,
    \item[b)] $d_2' = a \cdot d_1'$ for a set of representatives $a \in \mathbb{Z}$ of all elements $\overline a \in (\mathbb{Z}/M \mathbb{Z})^\times$.
\end{itemize}
Since $(\mathbb{Z}/M \mathbb{Z})^\times$ acts on itself transitively, these two operations can pairwise relate any two elements of the set $D'$ above. This proves the claim of the Lemma.

For our operations below, consider the multidegrees
\begin{align*}
    \underline{d}^\textup{can} &= (2g(v)-2+\mathrm{val}(v))_{v \in V(\Gamma)} \in \Pic^{2g-2}(\Gamma)\,,&\\
    \underline{d}^{(k)} &= (n_k(v))_{v \in V(\Gamma)} \in \Pic^{n_k}(\Gamma)& k=1,\dots,p
\end{align*}where $\mathrm{val}(v)$ is the number of (non-leg) half-edges at $v$ and $n_k(v)$ is the number of legs labeled ``$k$'' at $v$. Both of them are clearly $\Aut(\Gamma)$-invariant. 

For operation a) we simply observe that, by B\'ezout's identity a suitable linear combination of translations in $\mathrm{Pic}(\Gamma)$ by $\underline{d}^\textup{can}$ and $\underline{d}^{(k)}$, $k=1,\dots,p$, 
has the desired property. 

For operation b) consider an arbitrary element of $(\mathbb{Z}/M \mathbb{Z})^\times$ represented by a number $0 \leq \overline a \leq M$ with $\gcd(\overline a, M)=1$. By Dirichlet's theorem about primes in arithmetic progressions, the sequence $\overline a + m M$ for $m \in \mathbb{Z}_{\geq 0}$ contains infinitely many prime numbers. Choose such an odd prime number $a$ which is larger than the cardinality $|\Pic^0(\Gamma)|$ of the Jacobian of the graph $\Gamma$. Then we claim that the componentwise multiplication
\[
\phi_a : \Pic^{d_1}(\Gamma) \to \Pic^{a \cdot d_1}(\Gamma), \quad \underline{d} \mapsto a \cdot \underline{d} = (a \cdot \underline{d}(v))_{v \in V(\Gamma)}
\]
defines an $\Aut(\Gamma)$-equivariant bijection. First, it is well defined since it sends twist multidegrees to themselves and thus descends to the Picard group of $\Gamma$. It is also clearly equivariant with respect to graph automorphisms. For the bijectivity, choose an arbitrary element $\underline{d}_0 \in \Pic^{d_1}(\Gamma)$. Then by the classical theory of divisors on graphs, the map
\[
t_{\underline{d}_0} : \Pic^0(\Gamma) \to \Pic^{d_1}(\Gamma), \quad \underline{d} \mapsto \underline{d} + \underline{d}_0
\]
is a bijection. But the composition $\phi_a \circ t_{\underline{d}_0}$ is clearly equal to the map $t_{a \cdot \underline{d}_0} \circ m_a$, where
\[
m_a : \Pic^0(\Gamma) \to \Pic^0(\Gamma), \quad\underline{d} \mapsto a \cdot \underline{d}
\]
is the multiplication by $a$. But by assumption, the number $a$ is prime and larger than the cardinality of the finite abelian group $\Pic^0(\Gamma)$. Thus $m_a$ is bijective, finishing the proof that $\phi_a$ is likewise bijective.

The above operation $\phi_a$ relates $d_1$ to the degree $d_2=a \cdot d_1$. Expressed in coordinates $d'$ this means:
\[
d_2 = a \cdot d_1 \implies d_2' = d_2 - g+1 = a d_1 - g+1 = a (d_1' + g -1) -g+1 = a d_1' + (a-1) (g-1)
\]
To obtain the desired operation b) relating $d_1'$ to $d_2' = a \cdot d_1'$ we can simply compose $\phi_a$ by the translation operation with $-(a-1)/2 \cdot \underline{d}^\textup{can}$, noting that the prime number $a$ is odd.
\end{proof}

We can now prove \cref{combinatorial claim}.

\begin{proof}[Proof of \cref{combinatorial claim}]
In the case $g=0$, both sets which are claimed to admit an equivariant bijection are singleton sets (because in $g=0$ there is a unique stable multidegree for every nondegenerate stability condition). Thus the statement is trivial in this case, and in the rest of the proof we may assume $g \geq 1$.

By condition \eqref{condition1} of a Pagani--Tommasi stability condition, the maps
\[ \{ \underline d : (\Gamma,\Gamma_0,\underline d) \in \sigma \} \longrightarrow \mathrm{Pic}^{d-e_{\Gamma,\Gamma_0}}(\Gamma_0) \]
and 
\[ \{ \underline d' : (\Gamma,\Gamma_0,\underline d') \in \sigma'\} \longrightarrow \mathrm{Pic}^{d'-e_{\Gamma,\Gamma_0}}(\Gamma_0) \]
are both bijections, for every pair $(\Gamma,\Gamma_0)$. They are also equivariant with respect to graph automorphisms, where $\mathrm{Aut}(\Gamma,\Gamma_0)$ acts on the target via the projection to $\mathrm{Aut}(\Gamma_0)$. Thus it suffices to show that there exists an $\mathrm{Aut}(\Gamma,\Gamma_0)$-equivariant bijection between $\mathrm{Pic}^{d-e_{\Gamma,\Gamma_0}}(\Gamma_0)$ and $\mathrm{Pic}^{d'-e_{\Gamma,\Gamma_0}}(\Gamma_0)$. 

We claim that this follows from \cref{Lem:dmultiplication}. Let $\Gamma_0'$ denote the stable graph with $n+2e_{\Gamma,\Gamma_0}$ legs obtained from $\Gamma_0$ by ``cutting'' all the edges of $\Gamma$ not in $\Gamma_0$. Since $\Gamma_0'$ is just obtained from $\Gamma_0$ by attaching further legs, we have $\mathrm{Pic}(\Gamma_0)=\mathrm{Pic}(\Gamma_0')$, naturally with respect to the action of the group $\mathrm{Aut}(\Gamma,\Gamma_0)$. Hence it suffices to show that there is an $\mathrm{Aut}(\Gamma_0')$-equivariant bijection between $\mathrm{Pic}^{d-e_{\Gamma,\Gamma_0}}(\Gamma_0')$ and $\mathrm{Pic}^{d'-e_{\Gamma,\Gamma_0}}(\Gamma_0')$. We consider $\Gamma_0'$ as having groups of $n_1,\dots,n_p,2e_{\Gamma,\Gamma_0}$ indistinguishable legs.

We assumed that $\sigma$ and $\sigma'$ were invariant under $S_\lambda$. According to \cite[Remark 7.7]{fpv}, such invariant Pagani--Tommasi stability conditions exist only for $d$ and $d'$ satisfying the condition
 \[ \gcd(g-1+d,2g-2,n_1,\dots,n_p)=1.\]
 The condition we need to verify, to apply \cref{Lem:dmultiplication}, is that 
 \[ \gcd(g(\Gamma_0')-1+(d-e_{\Gamma,\Gamma_0}),2g(\Gamma_0')-2,n_1,\dots,n_p,2e_{\Gamma,\Gamma_0})=1.\]
 The result follows once we note that $g(\Gamma_0') = g - e_{\Gamma,\Gamma_0}$. 
\end{proof}


\subsection{Computation} \label{compute}
	The results of this paper do not only prove that the cohomology groups $H^*(\Jbar_{g,n}^d(\sigma),\Q)$ are independent of $d$ and $\sigma$, they also provide an algorithm for computing these cohomology groups, assuming that we know all of the equivariant Euler characteristics $\chi_c^{S_{n'}}(\Jcal_{g',n'})$ of the uncompactified moduli spaces, with $g' \leq g$, and $2g' - 2 + n' \leq 2g-2+n$. 
	

    \vspace{6pt}
    \noindent $\bullet$ How does one compute the Hodge numbers of $\Jbar_{g,n}^d(\sigma)$?
    \vspace{6pt}
    
    What we need to do is to enumerate all isomorphism classes of pairs $(\Gamma,\Gamma_0)$, where $\Gamma$ is a stable graph of type $(g,n)$, and $\Gamma_0$ is a connected subgraph containing all vertices, and for each such pair calculate the automorphism group $\Aut(\Gamma,\Gamma_0)$ and the Picard torsor $\mathrm{Pic}^{d-e_{\Gamma,\Gamma_0}}(\Gamma_0)$. For each pair $(\Gamma,\Gamma_0)$, consider the union of the corresponding strata in $\Jbar_{g,n}^d(\sigma)$. The compactly supported cohomology of this union has the same class in $K_0(\mathsf{MHS})$ as
    \[ \Big( H^*_c(T_{\Gamma_0},\Q) \otimes \Q[\mathrm{Pic}^{d-e_{\Gamma,\Gamma_0}}(\Gamma_0) ] \otimes \bigotimes_{v \in V(\Gamma)} H^*_c(\Jcal_{g(v),n(v)},\Q)  \Big)^{\Aut(\Gamma,\Gamma_0)}.\]
	Here $\Q[\mathrm{Pic}^{d-e_{\Gamma,\Gamma_0}}(\Gamma_0) ]$ denotes the $\Q$-vector space spanned by the finite set $\mathrm{Pic}^{d-e_{\Gamma,\Gamma_0}}(\Gamma_0)$, thought of as a trivial Hodge structure of weight $0$, but with a nontrivial action of the group $\Aut(\Gamma,\Gamma_0)$. From this, we can compute the Hodge--Deligne polynomial of each such union of strata in terms of the $S_{n(v)}$-equivariant Hodge--Deligne polynomials of the factors $\Jcal_{g(v),n(v)}$; we need the equivariance to know how $\Aut(\Gamma,\Gamma_0)$ acts. 

    It would be interesting if this procedure could be carried out systematically in terms of symmetric functions, akin to Getzler--Kapranov's \cite[Theorem 8.13]{getzlerkapranov} (or the approach of Kannan--Song \cite{kannansong}), which solves the analogous problem when there is no $\Gamma_0$, and the factors $H^*_c(T_{\Gamma_0},\Q) \otimes \Q[\mathrm{Pic}^{d-e_{\Gamma,\Gamma_0}}(\Gamma_0) ]$ are omitted.

    \vspace{6pt}
    \noindent $\bullet$ How does one compute $\chi_c^{S_{n}}(\Jcal_{g,n})$?
    \vspace{6pt}

     Suppose first $g \geq 2$. The forgetful map $\Jcal_{g,n} \to \Mcal_g$ is a fiber bundle with fiber $F(X,n) \times J(X)$, where $X$ is a fixed curve of genus $g$, and $F(X,n)$ denotes the configuration space of $n$ distinct ordered points on $X$. We may think of the cohomologies $H^*_c(F(X,n),\Q)$ and $H^*_c(J(X),\Q)$ as variations of mixed Hodge structure on $\Mcal_g$. Let us write $\chi_{\Mcal_g, c}^{S_n}(F(X,n))$ for the corresponding Euler characteristic, now taken in the Grothendieck group of polarizable variations of Hodge structure over $\Mcal_g$ with an action of $S_n$. Getzler (see \cite[Theorem 4.5]{resolving} or \cite{mixedhodge}) has given an explicit formula for $\chi_{\Mcal_g, c}^{S_n}(F(X,n))$ in terms of standard variations of Hodge structure $\mathbb V_\lambda$ on $\Mcal_g$ associated to irreducible representations of the symplectic group, indexed by partitions. This formula for $\chi_{\Mcal_g, c}^{S_n}(F(X,n))$ includes precisely those $\mathbb V_\lambda$ with $|\lambda| \leq n$. The Euler characteristic $\chi_{\Mcal_g,c}(J(X))$ is straightforwardly $\sum_k (-1)^k [\bigwedge^k \mathbb V_1]$. Multiplying these two together gives an expression for 
	\[ \chi_{\Mcal_g, c}^{S_n}(F(X,n) \times J(X))\]
in terms of classes of local systems $\mathbb V_\mu$ with $|\mu| - \ell(\mu) \leq n$. (One can show that these $\mu$ are precisely the representations of the symplectic group which occur as a summand in $\mathbb V_\lambda \otimes \bigwedge^k \mathbb V_1$ for some $\lambda$ with $|\lambda| \leq n$ and some $k$.) We then obtain a formula for $ \chi_c^{S_n}(\Jcal_{g,n})$
by ``integrating over $\Mcal_g$'', i.e.~substituting each occurrence of $\mathbb V_\lambda$ in $\chi_{\Mcal_g, c}^{S_n}(F(X,n) \times J(X))$ for the Euler characteristic of its cohomology, $\chi_c(\Mcal_g,\mathbb V_\lambda)$. In particular, $\chi_c^{S_n}(\Jcal_{g,n})$ can be expressed as a sum of terms $\chi_c(\Mcal_g,\mathbb V_\lambda)$ with $|\lambda| - \ell(\lambda) \leq n$.

The situation when $g=1$ is similar and even a little simpler: the forgetful map $\mathcal J_{g,n} \to \mathcal M_{1,1}$ has fiber $F(E,n)$, where $E$ is an elliptic curve, and one similarly expresses $\chi_c^{S_n}(\mathcal J_{g,n})$ in terms of $\chi_c(\Mcal_{1,1},\mathbb V_\lambda)$. When $g=0$, $\mathcal J_{0,n}= \mathcal M_{0,n}$, whose cohomology was calculated $S_n$-equivariantly by Getzler \cite{genuszero}.

The cohomology groups $H^*_c(\mathcal M_{1,1},\mathbb V_\lambda)$ are calculated classically by Eichler and Shimura. See Zucker \cite{zuckerdegeneratingcoefficients} or \cite[Theorem 5.3]{resolving} for the mixed Hodge structure. In genus two, a formula for $\chi_c(\mathcal A_2,\mathbb V_\lambda)$ was conjectured in \cite{fvdg1} and proven in \cite{localsystemsA2}, where in fact $H^*_c(\mathcal A_{2},\mathbb V_\lambda)$ is calculated (as a real Hodge structure, or $\ell$-adic Galois representation). Using the branching rule of \cite[Proposition 3.4]{d-elliptic} one can then compute $\chi_c(\Mcal_2,\mathbb V_\lambda)$ for any $\lambda$. In genus $3$, $\chi_c(\mathcal M_3,\mathbb V_\lambda)$ is known in the Grothendieck group of Galois representations unconditionally for $|\lambda| \leq 7$ \cite{Bergstrom_genus_three}, and for $|\lambda| \leq 14$ conditionally on standard conjectures in the Langlands program \cite{chenevierlannes}. When $|\lambda|$ is even, one can calculate $\chi_c(\mathcal M_3,\mathbb V_\lambda)$ in terms of $\chi_c(\mathcal A_3,\mathbb V_\lambda)$. A formula for the latter was conjectured in \cite{bfvdg} and proven in \cite{taibi}. In genus $4$, $\chi_c(\Mcal_4,\mathbb V_\lambda)$ is known for $|\lambda| \leq 3$ \cite{bfp}, which, however, is not enough information to determine even $\chi_c(\mathcal J_4)$.

See also \cite[Proposition E]{kannan2026virtualhodgenumbersmathcalmg} for related work, expressing the Euler characteristics $\chi_c^{S_n}(\mathcal J_{g,n})$ in terms of the Euler characteristics $\chi_c^{S_n}(\mathcal C_{g}^n)$, for fixed $g$, where $\mathcal C_{g}^n$ is the $n$-fold fibered power of the universal curve over $\mathcal M_g$. The weight-zero and integer-valued Euler characteristics of $\mathcal J_{g,n}$ were studied in \cite{kannan}.

\begin{appendix}
\section{Geometry of the universal Jacobian over
\texorpdfstring{$\mathcal{M}_g$}{M\_g}} \label{Sect:Misc}


The following result was first explained to us by Jeremy Feusi. The  referee pointed out that the key step of the proof appeared already in \cite[Corollary 6.1]{KP19}, with essentially the same argument; the proof given here has been shortened accordingly.

\begin{theorem}
For $g \geq 1$ and $2g-2+n>0$, the universal Jacobians 
$\mathcal{J}_{g,n}^d$ and
$\mathcal{J}_{g,n}^{d'}$ are $S_n$-equivariantly isomorphic over $\mathcal{M}_{g,n}$
if and only if 
$$d\equiv d' \quad \text{or}\quad -d\equiv d'
\mod \gcd(2g-2,n)\, .$$
\end{theorem}
\begin{proof}
The case $n=0$ was covered in \cite[Lemma 8.1]{Cap94}, so from now on we assume $n > 0$. Then, to prove the claim, we will use the result of \cite[Corollary 6.1]{KP19}:

($\dagger$) Every isomorphism $\varphi: \mathcal{J}_{g,n}^d \to \mathcal{J}_{g,n}^{d'}$ over $\mathcal{M}_{g,n}$ is of the form
\begin{equation*}
    T_{\textbf{a}, \epsilon}(C, p_1, \ldots, p_n, \mathcal{L}) = (C, p_1, \ldots, p_n, \mathcal{L}^{\otimes \epsilon} \otimes_{\mathcal{O}_C} \omega_C^{\otimes a}(\sum_{i=1}^n a_i p_i))\,,
\end{equation*}
where $\textbf{a} = (a, a_1, \ldots, a_n) \in \mathbb{Z}^{n+1}$ and $\epsilon \in \{\pm 1\}$. 

We will  argue that $T_{\textbf{a},\epsilon}$ being $S_n$-equivariant then 
forces all $a_1=\ldots = a_n$ to be equal.  Indeed, we must have
\begin{equation} \label{eqn:Feusi1}
\mathcal{L}^{\otimes \epsilon} \otimes_{\mathcal{O}_C}\omega_C^{\otimes a}(\sum_{i=1}^n a_i p_i)\cong \mathcal{L}^{\otimes \epsilon} \otimes_{\mathcal{O}_C}\omega_C^{\otimes a}(\sum_{i=1}^n a_i p_{\tau(i)})
\end{equation}
for any permutation $\tau\in S_n$, since the action $S_n \curvearrowright \mathcal{J}_{g,n}^d$ leaves the line bundle unchanged. The equality \eqref{eqn:Feusi1} is equivalent to
\begin{equation}
    \mathcal{O}_C \cong \mathcal{O}_C(\sum_{i=1}^n a_i(p_i-p_{\tau(i)}))\,.
\end{equation}
For $i \geq 2$, let $\tau$ be the transposition of $1$ and $i$. Then, we get the equality 
$$\mathcal{O}_C \cong \mathcal{O}_C((a_1-a_i)(p_i-p_1))\, .$$
Since $g \geq 1$, generically $\mathcal{O}_C(p_i-p_1)$ is not a torsion point. Hence, $a_1=a_i$. 
Writing $a_1=\cdots=a_n=b$, such a transformation has relative degree
\[
   \epsilon d + a(2g-2)+bn .
\]
Hence an $S_n$-equivariant isomorphism of the above form from degree $d$ to degree $d'$ can exist only if
\[
   d'-\epsilon d \in (2g-2)\mathbb Z+n\mathbb Z,
\]
or equivalently
\[
   d' \equiv \epsilon d \pmod{\gcd(2g-2,n)}.
\]
Conversely, if this congruence holds, choose integers $a,b$ such that
\[
   d'-\epsilon d = a(2g-2)+bn.
\]
Then $T_{\mathbf a,\epsilon}$ with $a_1=\cdots=a_n=b$ gives an $S_n$-equivariant isomorphism. 
This proves the stated criterion.
\end{proof}


\begin{theorem}[Qizheng Yin]
\label{thm:Jac-degree-distinguishes-K0}
Fix an integer $g\ge 22$ and let $d,\widehat d\in \mathbb{Z}$. Then{\footnote{For Theorem \ref{thm:Jac-degree-distinguishes-K0}, $[\mathcal{J}_g^{d}]$ denotes the class of the coarse moduli space in the Grothendieck group of varieties $K_0(\mathrm{Var}_\C)$.}}
\[
[\mathcal{J}_g^d] = [\mathcal{J}_g^{\widehat d}]
\quad\text{in}\quad
K_0(\mathrm{Var}_\C)
\]
if and only if $
\widehat d \equiv \pm d \mod{2g-2}$.
\end{theorem}

For the proof of Theorem \ref{thm:Jac-degree-distinguishes-K0}, we will use a result by Larsen--Lunts \cite[Section 2.5]{LarsenLunts2003}: {\em there is a ring homomorphism
\[
\Phi_{\mathrm{sb}}:K_0(\mathrm{Var}_\C)\longrightarrow \mathbb{Z}[\mathrm{SB}]
\]
to the monoid ring on stable birational equivalence classes of smooth projective varieties $X$, which sends the class $[X] \in K_0(\mathrm{Var}_\C)$ of such a smooth projective variety to $1 \cdot (X) \in \mathbb{Z}[\mathrm{SB}]$}. To extend it to arbitrary varieties $U$, we choose a smooth projective birational model $X$ of $U$ with a common open $V \subset U, V \subset X$, and obtain $$\Phi_{\mathrm{sb}}([U]) = (X) - \Phi_{\mathrm{sb}}([X \setminus V]) + \Phi_{\mathrm{sb}}([U \setminus V])\, .$$ The last two terms above are defined inductively (since they are of smaller dimension). In particular, we obtain an expression
\begin{equation} \label{eq:Phisb}
\Phi_{\mathrm{sb}}([U]) \;=\; (X) \;+\; \sum_i n_i (V_i) \quad\text{in}\quad
\mathbb Z[\mathrm{SB}]\,,
\end{equation}
where each $V_i$ is smooth projective with $\dim V_i < \dim X$.

To apply the Larsen--Lunts construction in the proof of \cref{thm:Jac-degree-distinguishes-K0}, we will require the following result.

\begin{lemma}[No cancellation of the top stable birational class for $\kappa\ge 0$]
\label{lem:no-cancellation-phi-sb}
Let $U$ be an irreducible complex variety, and let $X$ be a smooth projective birational model of $U$.
Assume $\kappa(X)\ge 0$ (or weaker, that $X$ is not uniruled).
Then, no term $(V_i)$ in the expression \eqref{eq:Phisb} is equal to $(X)$ in $\mathrm{SB}$, hence the coefficient of $(X)$ is well-defined and equal to $1$.

If $U'$ is another irreducible variety admitting a smooth projective birational model $X'$, with $\dim U = \dim U'$ and $\Phi_{\mathrm{sb}}([U])=\Phi_{\mathrm{sb}}([U'])$ in $\mathbb Z[\mathrm{SB}]$,
then $(X)=(X')$.
Moreover, $X$ and $X'$ are birational.
\end{lemma}

\begin{proof}
We first show that $(X)$ cannot coincide with any $(V_i)$.
Assume for contradiction that $(X)=(V_i)$ in $\mathrm{SB}$ for some $i$. Set $d=\dim X$ and $e=\dim V_i$. 
By construction $e<d$, and we write $r=d-e>0$. 
The equality $(X)=(V_i)$ shows that $X$ and $V_i \times \mathbb{P}^r$ are smooth projective varieties that are stably birational and of the same dimension. But since $X$ is non-uniruled, Liu--Sebag's \cite[Corollary 1(1)]{LiuSebag2010} yields that $X$ and $V_i \times \mathbb{P}^r$ are actually birational. This contradicts the uniruledness of $V_i \times \mathbb{P}^r$.

Consequently no cancellation of the basis element $(X)$ occurs in $\mathbb Z[\mathrm{SB}]$,
and the coefficient of $(X)$ in $\Phi_{\mathrm{sb}}([U])$ is well-defined and equal to $1$.

For the final assertion, write (using \eqref{eq:Phisb}) expansions
\[
\Phi_{\mathrm{sb}}([U])=(X)+\sum_i n_i(V_i)
\quad\text{and}\quad
\Phi_{\mathrm{sb}}([U'])=(X')+\sum_j m_j(W_j)\,,
\]
where $\dim V_i, \dim W_j<\dim X = \dim X'$ for all $i,j$. Moreover, we have $(X) \neq (V_i)$ for all $i$, and repeating the argument above we also have $(X) \neq (W_j)$ for all $j$.
Now recall that $\mathbb Z[\mathrm{SB}]$ is, by definition, the free abelian group on the set
$\mathrm{SB}$ (equivalently, the monoid ring of $\mathrm{SB}$), so every element has a unique
expression as a finite $\mathbb Z$-linear combination of the basis symbols $(T) \in  \mathrm{SB}$.
Thus, if $\Phi_{\mathrm{sb}}([U])=\Phi_{\mathrm{sb}}([U'])$, then the coefficient of the basis
element $(X)$ on the left (which is $1$) must equal its coefficient on the right.
Hence it must be
that $(X)=(X')$ in~$\mathrm{SB}$, so\ $X$ and $X'$ are stably birational.

Finally, since $X$ and $X'$ are smooth projective varieties that are stably birational and of the same dimension, and since $X$ is non-uniruled, \cite[Corollary 1(1)]{LiuSebag2010} again yields that $X$ and $X'$ are birational.
\end{proof}

\begin{proof}[Proof of \cref{thm:Jac-degree-distinguishes-K0}]
If $\widehat d\equiv \pm d \mod{2g-2}$, then $\mathcal{J}_g^d$ and $\mathcal{J}_g^{\widehat d}$ are easily seen to be isomorphic; hence $[\mathcal{J}_g^d]=[\mathcal{J}_g^{\widehat d}]$ in $K_0(\mathrm{Var}_\C)$. The other direction is done in two steps:

\vspace{8pt}
\noindent\textbf{Step 1: apply $\Phi_{\mathrm{sb}}$ and use Lemma \ref{lem:no-cancellation-phi-sb}.}
Applying the ring homomorphism $\Phi_{\mathrm{sb}}$ to the assumed equality
$[\mathcal{J}_g^d]=[\mathcal{J}_g^{\widehat d}]$ in $K_0(\mathrm{Var}_\C)$ gives
\[
\Phi_{\mathrm{sb}}\bigl([\mathcal{J}_g^d]\bigr)
=
\Phi_{\mathrm{sb}}\bigl([\mathcal{J}_g^{\widehat d}]\bigr)
\quad\text{in}\quad
\mathbb Z[\mathrm{SB}]\,.
\]
Let $X$ and $X'$ be smooth projective birational models of $\mathcal{J}_g^d$ and $\mathcal{J}_g^{\widehat d}$,
respectively. By \cite[Theorem~1.3]{viv}, for $g\ge 22$, we have 
$$\kappa(X)(=\kappa(X'))=3g-3\ge 0\,.$$
Moreover, $\dim(\mathcal{J}_g^d)=\dim(\mathcal{J}_g^{\widehat d}) =4g-3$.
Therefore Lemma~\ref{lem:no-cancellation-phi-sb} applies to
$U=\mathcal{J}_g^d$ and $U'=\mathcal{J}_g^{\widehat d}$ and yields that $X$ and $X'$ are birational.
In particular, $\mathcal{J}_g^d$ and $\mathcal{J}_g^{\widehat d}$ are birational.

\vspace{8pt}
\noindent\textbf{Step 2: birational classification.}
By \cite[Theorem~1.7]{viv}, for $g\ge 22$,  $\mathcal{J}_g^d$ is birational to
$\mathcal{J}_g^{\widehat d}$ if and only if $\widehat d\equiv \pm d \mod{2g-2}$. Therefore $[\mathcal{J}_g^d] =  [\mathcal{J}_g^{\widehat d}]$ in $K_0(\mathrm{Var}_\C)$ forces $\widehat d\equiv \pm d \mod{2g-2}$.
\end{proof}


\end{appendix}
\printbibliography
\end{document}